\documentclass[letterpaper,10pt,conference,romanappendices]{ieeeconf}
\IEEEoverridecommandlockouts                           

\usepackage{graphics} 
\usepackage{amsmath} 
\usepackage{amssymb}  
\usepackage{theorem}
\usepackage{color,xspace}
\usepackage{graphicx}
\usepackage{subcaption}
\usepackage[noadjust]{cite}
\usepackage{cleveref}
\usepackage{bm}
\usepackage{bbm}
\usepackage{tikz}
\usepackage{epstopdf} 
\usepackage{algorithm}
\usepackage{algpseudocode}
 \usepackage{booktabs}
 \let\labelindent\relax
\usepackage{enumitem}

\newtheorem{theorem}{Theorem}
\newtheorem{lemma}{Lemma}	

\newtheorem{remark}{Remark}
\newtheorem{example}{Example}

\newtheorem{definition}{Definition}	
\newtheorem{problem}{Problem}

\def\bx{\underline{\mathbf{x}}}

\def\bx{\mathbf{x}}

\def\by{\mathbf{y}}

\def\dom{\mathrm{dom}}

\title{\LARGE \bf Design of First-Order Optimization Algorithms via \\ Sum-of-Squares Programming}
\author{Mahyar Fazlyab, Manfred Morari, Victor M. Preciado
\thanks{Work supported by the NSF under grants CAREER-ECCS-1651433 and IIS-1447470. The authors are with the Department of Electrical and Systems Engineering, University of Pennsylvania. The authors are with the Department of Electrical and Systems Engineering, University of Pennsylvania. Email: \{mahyarfa,morari,preciado\}@seas.upenn.edu.}}

\begin{document}

\maketitle
\thispagestyle{empty}
\pagestyle{empty}

\begin{abstract}
	In this paper, we propose a framework based on sum-of-squares programming to design iterative first-order optimization algorithms for smooth and strongly convex problems. Our starting point is to develop a polynomial matrix inequality as a sufficient condition for exponential convergence of the algorithm. The entries of this matrix are polynomial functions of the unknown parameters (exponential decay rate, stepsize, momentum coefficient, etc.). We then formulate a polynomial optimization, in which the objective is to optimize the exponential decay rate over the parameters of the algorithm. Finally, we use sum-of-squares programming as a tractable relaxation of the proposed polynomial optimization problem. We illustrate the utility of the proposed framework by designing a first-order algorithm that shares the same structure as Nesterov's accelerated gradient method.
\end{abstract}

\section{Introduction}

Many applications in science and engineering involve solving convex optimization problems using iterative methods.
In the development
of iterative optimization algorithms, there are several objectives
that algorithm designers should consider \cite{wright1999numerical}: 
\begin{itemize}[leftmargin=*]
	\item Robust performance guarantees: Algorithms should perform well for
	a wide range of optimization problems in their class. 
	\item Time and space efficiency: Algorithms should be efficient in terms
	of both time and storage, although these two may conflict. 
	\item Accuracy: Algorithms should be able to provide arbitrarily accurate
	solutions to the problem at a reasonable computational cost. 
\end{itemize}
In general, these goals may conflict. For example, a rapidly convergent
method (e.g., Newton's method) may require too much computer storage
and/or computation. In contrast, a robust method, resilient to noise
and uncertainties, may also be too slow in finding an optimal solution.
Trade-offs between, for example, convergence rate and storage requirements,
or between robustness and speed, are central issues in numerical optimization \cite{wright1999numerical}.

In recent years, there has been an increasing interest in using tools from control theory to study the convergence properties of iterative optimization algorithms. The connection between control and optimization is made by interpreting these iterative algorithms as discrete-time dynamical systems with feedback. This interpretation provides many insights and new directions of research. In particular, we can use control tools to study disturbance rejection properties of optimization algorithms \cite{wang2009distributed,wang2011control,wang2010controlapproach}, study robustness to parameter and model uncertainty \cite{lessard2016analysis}, and analyze tracking and adaptation capabilities \cite{fazlyab2017prediction}. This interpretation also opens the door to the use of control tools for algorithm design \cite{lessard2016analysis,cyrus2017robust,fazlyab2017analysis}.


The main aim of this paper is to develop an optimization framework, based on tools from robust control theory and polynomial optimization, to design first-order optimization algorithms for the class of smooth and strongly convex problems, in which the convergence rate is exponential ($\mathcal{O}(\rho^{k})$ for some $0\leq {\rho}<1$). To this end, we start with a result in \cite{fazlyab2017analysis}, in which the authors derive a matrix inequality as a sufficient condition for exponential stability of the algorithm. The entries of this matrix are polynomial functions of \emph{(i)} the unknown parameters of the algorithm (e.g., stepsize, momentum coefficient) and \emph{(ii)} the unknown exponential decay rate. We then formulate a polynomial optimization problem in which the cost function is the exponential decay rate $\rho$, the constraint is the polynomial matrix inequality described above, and the decision variables are the tunable parameters of the algorithm.
%
%
Finally, we use sum-of-squares programming as a tractable relaxation of the polynomial optimization problem to tune the parameters of the algorithm. We illustrate our approach by designing a first-order method sharing the same structure as Nesterov's accelerated method. To the best of our knowledge, this is the first work on principled and computationally efficient numerical algorithm design.

The rest of the paper is organized as follows. In Section \ref{section: Algorithm Analysis}, we use a Lyapunov analysis framework, in which we cast the problem of finding the worst-case exponential
rate bound of a first-order optimization algorithm as a semidefinite program. In Section \ref{section: Algorithm Synthesis}, we turn to algorithm synthesis and use the results of Section \ref{section: Algorithm Analysis}
to formulate the algorithm design problem as a polynomial optimization and use sum-of-squares machinery to solve the algorithm design problem. Finally, we illustrate the performance of our design framework via numerical simulations.

\section{Algorithm Analysis} \label{section: Algorithm Analysis}

\subsection{Preliminaries}

We denote the set of real numbers by $\mathbb{R}$, the set of real $n$-dimensional vectors by $\mathbb{R}^n$, the set of $m\times n$-dimensional matrices by $\mathbb{R}^{m\times n}$, and the $n$-dimensional identity matrix by $I_n$. We denote by $\mathbb{S}^{n}$, $\mathbb{S}_{+}^n$, and $\mathbb{S}_{++}^n$ the sets of $n$-by-$n$ symmetric, positive semidefinite, and positive definite matrices, respectively. We denote by $\mbox{vec}(\cdot)$ a linear transformation which converts the matrix into a column vector. For a function $f \colon \mathbb{R}^d \to \mathbb{R}$, we denote by $\dom \, f=\{x \in \mathbb{R}^n \colon f(x)<\infty\}$ the effective domain of $f$. The $p$-norm ($p \geq 1$) is displayed by $\|\cdot\|_p \colon \mathbb{R}^d \to \mathbb{R}_{+}$. For two matrices $A \in \mathbb{R}^{m\times n}$ and $B\in \mathbb{R}^{p\times q}$ of arbitrary dimensions, we denote their Kronecker product by $A \otimes B$. We define $\mathbb{R}\left[\mathbf{x}\right]$ as the polynomial ring in $n$ variables and $\mathbb{R\left[\mathbf{x}\right]}_{n,d}$ as the polynomial in $n$ variables of degree at most $d$.

\begin{definition}[Smoothness] \normalfont
	A differentiable function $f \colon \mathbb{R}^d \to \mathbb{R}$ is $L_f$-smooth on $\mathcal{S} \subseteq \dom \, f$ if the following inequality
	\begin{subequations}
		\begin{align} \label{eq: Lipschitz continuity}
		\|\nabla f(x)-\nabla f(y)\|_2 \leq L_f \|x-y\|_2,
		\end{align}
		holds for some $L_f>0$ and all $x,y \in \mathcal{S}$. \eqref{eq: Lipschitz continuity} implies
		\begin{align} \label{eq: Lipschitz continuity function values}
		f(y) \leq f(x) + \nabla f(x)^\top (y-x) + \dfrac{L_f}{2} \|y-x\|_2^2,
		\end{align}
	\end{subequations}
	for all $x,y \in \mathcal{S}$.
\end{definition}

\begin{definition}[Strong Convexity] \normalfont
	A differentiable function $f \colon \mathbb{R}^d \to \mathbb{R}$ is $m_f$-strongly convex on $\mathcal{S} \subseteq \dom \, f$ if the following inequality
	\begin{subequations}
		\begin{align} \label{eq: strong convexity}
		m_f \|x-y\|_2^2 \leq (x-y)^\top (\nabla f(x)-\nabla f(y)),
		\end{align}
		holds for some $m_f>0$ and all $x,y \in \mathcal{S}$. An equivalent definition is that
		\begin{align} \label{eq: strong convexity function values}
		f(x)+\nabla f(x)^\top (y-x) + \dfrac{m_f}{2} \|y-x\|_2^2  \leq f(y),
		\end{align}
	\end{subequations}
	for all $x,y \in \mathcal{S}$.
\end{definition}
We denote the class of functions satisfying \eqref{eq: Lipschitz continuity} and \eqref{eq: strong convexity} by $\mathcal{F}(m_f,L_f)$. A differentiable function $f$ belongs to the class $\mathcal{F}(m_f,L_f)$ on $\mathcal{S}$ if and only if the inequality
\begin{subequations}
	\begin{align}\label{eq: strongly convex IQC}
	\begin{bmatrix}
	x\!-\!y \\ \nabla f(x)\!-\!\nabla f(y)
	\end{bmatrix}^\top \! Q_f \! \begin{bmatrix}
	x\!-\!y \\ \nabla f(x)\!-\!\nabla f(y)
	\end{bmatrix} \geq 0,
	\end{align}
	holds for all $x,y \in \mathcal{S}$ \cite{nesterov2013introductory,lessard2016analysis}, where the indefinite matrix $Q_f \in \mathbb{S}^{2d}$ is given by
	\begin{align} \label{eq: quadratic matrix}
	Q_f = \begin{bmatrix}
	-2m_fL_f & m_f+L_f \\ m_f+L_f & -2
	\end{bmatrix} \otimes I_d.
	\end{align}
\end{subequations}
The condition number of $f \in \mathcal{F}(m_f,L_f)$ is denoted by $\kappa_f = L_f/m_f$.

\subsection{Algorithm Representation}
Consider the following optimization problem
\begin{align} \label{eq: unconstrained problem}
\mathrm{minimize}_{x \in \mathbb{R}^d} \ f(x),
\end{align}
where $f \in \mathcal{F}(m_f,L_f)$ is smooth and strongly convex. Under this assumption, the well-known optimality condition of a point $y_{\star}$ is that
\begin{align*}
\nabla f(y_{\star}) =0.
\end{align*}
Note that $y_{\star}$ is unique as $f$ is strongly convex. Consider a first-order algorithm that generates a sequence of points $\{y_k\}$ that solves \eqref{eq: unconstrained problem}. In other words, we have $\lim\limits_{k \to \infty} f(y_k)=f(y_{\star})$, where $\nabla f(y_{\star})=0$. We can represent the algorithm in the following state-space form \cite{lessard2016analysis,fazlyab2017analysis}:
\begin{align} \label{eq: algorithm canonical form}
\xi_{k+1} &= A(\theta) \xi_k + B(\theta) \nabla f(y_k) \\ \nonumber
y_k &=C(\theta)\xi_k \\ \nonumber
x_k &=E(\theta)\xi_k, \nonumber 
\end{align}
where $\xi_{k} \in \mathbb{R}^n$ ($n \geq d$) is the state of the algorithm and $y_k \in \mathbb{R}^d$ is the output at which the gradient is evaluated. Furthermore, we assume that $x_k$ is another output whose fixed point is optimal, i.e., $\lim\limits_{k \to \infty} f(x_k)=f(x_{\star})$ where $x_{\star}=y_{\star}$. Therefore, the fixed points of the algorithm must naturally satisfy
\begin{align} \label{eq: algorithm fixed point}
\xi_{\star} &= A(\theta) \xi_{\star}, \quad y_{\star} = C(\theta) \xi_{\star}, \quad x_{\star} = E(\theta) \xi_{\star}, \quad x_{\star}=y_{\star}.
\end{align}
In particular, one of the eigenvalues of $A(\theta)$ is equal to one. Note that the matrices $A(\theta) \in \mathbb{R}^{n \times n}$, $B(\theta) \in \mathbb{R}^{n \times d}$, $C(\theta) \in \mathbb{R}^{d \times n}$, and $E(\theta) \in \mathbb{R}^{d \times n}$ in \eqref{eq: algorithm canonical form} are parameterized by the vector $\theta \in \mathbb{R}^p$, which is the concatenation of the parameters of the algorithm (e.g., stepsize, momentum coefficient, etc.). We give two examples below.
\begin{example} \label{example: gradient method}
	\normalfont The gradient method is given by the recursion
	\begin{align}
	\xi_{k+1} = \xi_{k} - h \nabla f(\xi_{k}), \ y_k = \xi_k, \ x_k = \xi_k,
	\end{align}
	where $h \geq 0$ is the stepsize. For this algorithm, $\theta = h$ is the only parameter, and the matrices $A(\theta),B(\theta),C(\theta)$ and $E(\theta)$ are given by
	\begin{align}
	A(\theta) = I_d, \quad B(\theta) = -h I_d, \quad C(\theta) = E(\theta) = I_d.
	\end{align}
\end{example}

\begin{example} \label{example: nesterov heavy-ball method}
	\normalfont Consider the following recursion defined on the two sequences $\{x_k\} $ and $\{y_k\} $,
	\begin{align} \label{eq: universal example}
	x_{k+1} &= x_k+\beta(x_k-x_{k-1})-h \nabla f(y_k), \\
	y_k &= x_k+\gamma(x_k-x_{k-1}), \nonumber 
	\end{align}
	where $h, \beta$ and $\gamma$ are nonnegative scalars. By defining the state vector $\xi_k = [x_{k-1}^\top \ x_k^\top]^\top \in \mathbb{R}^{2d}$ and the parameter vector $\theta = [h \ \beta \ \gamma]^\top$, we can represent \eqref{eq: universal example} in the canonical form \eqref{eq: algorithm canonical form}, as follows,
	\begin{align} \label{eq: accelerated proximal gradient method 11}
	\xi_{k+1} &= \begin{bmatrix} 0 & I_d \\ -\beta I_d & (\beta+1)I_d \end{bmatrix}\xi_k  + \begin{bmatrix} 0 \\ -h I_d\end{bmatrix} \nabla f(y_k) \\ 
	y_k &= \begin{bmatrix} -\gamma I_d & (\gamma+1)I_d\end{bmatrix} \xi_{k}. \nonumber
	\end{align}
	Therefore, the matrices $A(\theta),B(\theta),C(\theta)$ and $E(\theta)$ are given by
	\begin{align} \label{eq: nesterov heavy-ball matrices}
	\left[
	\begin{array}{c|c}
	A(\theta) & B(\theta) \\
	\hline
	C(\theta) & 0
	\end{array}
	\right] &= \left[
	\begin{array}{c|c}
	\begin{matrix}
	0 &I_d \\ -\beta I_d & (\beta+1)I_d
	\end{matrix} & \begin{matrix} 0 \\ -hI_d\end{matrix} \\
	\hline
	\begin{matrix} -\gamma I_d & (\gamma+1)I_d\end{matrix} & 0
	\end{array}
	\right] \\ \nonumber
	E(\theta) &= \begin{bmatrix}
	0 & I_d
	\end{bmatrix}^\top.
	\end{align}
	Notice that depending on the selection of $\beta$ and $\gamma$, \eqref{eq: universal example} describes various existing algorithms. For example, the gradient method corresponds to the case $\beta=\gamma=0$ (see Example \ref{example: gradient method}). In Nesterov's accelerated method, we have $\beta=\gamma$ \cite{nesterov1983method}. Finally, we recover Heavy-ball method by setting $\gamma=0$ \cite{POLYAK19641}. In Table \ref{table}, we provide various parameter selections and convergence rates for the gradient method, the heavy-ball method, and Nesterov's accelerated method \cite{nesterov1983method,POLYAK19641,lessard2016analysis}.
\end{example}
%

\begin{table*}
	\flushleft
	\centering
	\caption{\small Exponential decay rate of the algorithm in \eqref{eq: universal example} for various parameter selections.}
	\begin{tabular}{|l|l|l|}
		\hline
		Algorithm                                                            & Parameters                                                                                                                                                    & Exponential Decay Rate 
		\\ \hline
		\begin{tabular}[c]{@{}l@{}}Gradient\\ (Strongly Convex)\end{tabular} & $h=\frac{1}{L_f}$,  $\beta=\gamma=0$  & $\rho=1-\frac{1}{\kappa_f}$                 \\ \hline
		\begin{tabular}[c]{@{}l@{}}Gradient\\ (Strongly Convex)\end{tabular} & $h=\frac{2}{m_f+L_f}$, $\beta=\gamma=0$  & $\rho=\frac{\kappa_f-1}{\kappa_f+1}$  \\ \hline
		\begin{tabular}[c]{@{}l@{}}Nesterov\\ (Strongly Convex)\end{tabular} & $h=\frac{1}{L_f}$, $\beta=\gamma=\frac{\sqrt{\kappa_f}-1}{\sqrt{\kappa_f}+1}$ & $\rho = \sqrt{1-\frac{1}{\sqrt{\kappa_f}}}$ \\ \hline
		\begin{tabular}[c]{@{}l@{}}Nesterov\\ (Quadratics)\end{tabular}      & \begin{tabular}[c]{@{}l@{}}$h = \frac{4}{3L_f+m_f}$, $\beta = \frac{\sqrt{3\kappa_f+1}-2}{\sqrt{3 \kappa_f+1}+2}$\end{tabular}                          & $\rho = 1 - \frac{2}{\sqrt{3\kappa_f+1}}$            \\ \hline
		\begin{tabular}[c]{@{}l@{}}Heavy-ball\\ (Quadratics)\end{tabular}    & \begin{tabular}[c]{@{}l@{}}$h = \frac{4}{(\sqrt{L_f}+\sqrt{m_f})^2}$, $\gamma= \left(\frac{\sqrt{\kappa_f}-1}{\sqrt{\kappa_f}+1}\right)^2$,  $\beta=0$\end{tabular} & $\rho = \frac{\sqrt{\kappa_f}-1}{\sqrt{\kappa_f}+1}$ \\ \hline
	\end{tabular}
	\label{table}
\end{table*}

\subsection{Exponential Convergence via SDPs}
To measure the progress of the algorithm in \eqref{eq: algorithm canonical form} towards optimality, we make use of the following Lyapunov function \cite{fazlyab2017analysis}:
\begin{align} \label{eq: Lyapunov function}
V_k = f(x_k)-f(x_{\star}) + (\xi_k-\xi_{\star})^\top P (\xi_k-\xi_{\star}),
\end{align}
where $x_{\star}$ is the unique minimizer of $f$ and $P \in \mathbb{S}_{+}^n$ is an unknown positive \emph{semidefinite} matrix that does not depend on $k$. The first term in \eqref{eq: Lyapunov function} is the suboptimality of $x_k$ and the second term is a weighted ``distance" between the state $\xi_k$ and the fixed point $\xi_{\star}$. Notice that by this definition, $V_k$ is positive everywhere and zero at optimality. Suppose we select $P$ such that the Lyapunov function satisfies the inequality
\begin{align} \label{eq: Lyapunov exponential decay}
V_{k+1} \leq \rho^2 V_k \quad \mbox{for all } k,
\end{align}
for some $0\leq \rho<1$. By iterating down \eqref{eq: Lyapunov exponential decay} to $k=0$, we can conclude that $V_k \leq \rho^{2k} V_0$ for all $k$. This implies that
\begin{align} \label{eq: convergence result}
0 \leq f(x_k) - f(x_{\star}) &\leq \rho^{2k} V_0. 
\end{align}
In other words, the algorithm exhibits an $\mathcal{O}(\rho^{2k})$ convergence rate--in terms of objective values-- if the Lyapunov function satisfies the decrease condition \eqref{eq: Lyapunov exponential decay}. In the following result, developed in \cite{fazlyab2017analysis}, we present a matrix inequality whose feasibility implies \eqref{eq: Lyapunov exponential decay} and, hence, the exponential convergence of the algorithm. For notational convenience, we drop the argument $\theta$ wherever the dependence of the matrices $A(\theta),B(\theta),C(\theta),$ and $E(\theta)$ on $\theta$ is clear from the context.

\begin{theorem} \label{thm: exponential convergence}
	Let $x_{\star} = \mathrm{argmin}_{x} \ f(x)$, where $f \in \mathcal{F}(m_f,L_f)$. Consider the dynamical system in \eqref{eq: algorithm canonical form}, whose fixed point satisfies \eqref{eq: algorithm fixed point}. Define the matrices
	\begin{subequations} \label{eq: LMI matrices}
		\begin{align}
		M^0 &=\begin{bmatrix}
		    A^\top P A \!-\! \rho^2 P& A^\top P B\\  B^\top P A& B^\top P B \end{bmatrix} \\
		M^1 &=N^1+N^2\\
		M^2 &=N^1+N^3\\
		M^3 &=N^4, 
		\end{align}
	\end{subequations}
	where
	\begin{align*}
	N^{1} \!&= \begin{bmatrix}
	E A\!-\! C & EB \\ 0 & I_d 
	\end{bmatrix}^\top
	\begin{bmatrix}
	\frac{L_f}{2} I_d & \frac{1}{2} I_d \\ \frac{1}{2} I_d & 0
	\end{bmatrix}
	\begin{bmatrix}
	E A\!-\! C & EB \\ 0 & I_d 
	\end{bmatrix}  \nonumber \\ \nonumber
	N^{2} &=\begin{bmatrix}
	C-E & 0 \\ 0 & I_d
	\end{bmatrix}^\top \begin{bmatrix}
	-\frac{m_f}{2} I_d & \frac{1}{2} I_d \\ \frac{1}{2} I_d & 0
	\end{bmatrix} \begin{bmatrix}
	C-E & 0 \\ 0 & I_d
	\end{bmatrix}  \\ \nonumber 
	N^{3} &= \begin{bmatrix}
	C & 0 \\ 0 & I_d
	\end{bmatrix}^\top \begin{bmatrix}
	-\frac{m_f}{2} I_d & \frac{1}{2} I_d \\ \frac{1}{2} I_d & 0
	\end{bmatrix} \begin{bmatrix}
	C & 0 \\ 0 & I_d
	\end{bmatrix} \\
	N^{4} &=\begin{bmatrix}
	C & 0 \\ 0 & I_d
	\end{bmatrix}^\top \begin{bmatrix}
	\frac{-m_fL_f}{m_f+L_f} I_d & \frac{1}{2} I_d \\ \frac{1}{2} I_d  & \frac{-1}{m_f+L_f} I_d 
	\end{bmatrix} \begin{bmatrix}
	C & 0 \\ 0 & I_d
	\end{bmatrix}. \nonumber
	\end{align*} 
	Suppose there exists a positive semidefinite $P \in \mathbb{S}_{+}^n$, and nonnegative scalars $0<\rho\leq1$, $\lambda \geq 0$ that satisfy the following matrix inequality:
	\begin{align} \label{eq: LMI}
	M(\theta,\rho,\lambda,P):= M^0+\!\rho^2 M^1 \!+\! (1\!-\!\rho^2) M^2 \!+\! \lambda M^3 \preceq 0.
	\end{align}
	Then the sequence $\{x_k\}$ satisfies
	\begin{align} \label{eq: convergence rate}
	0 \leq f(x_k)-f(x_{\star}) \leq \rho^{2k} V_0 \quad \mbox{for all } k.
	\end{align}
\end{theorem}
\begin{proof}
	See \cite{fazlyab2017analysis}.
\end{proof}

According to Theorem \ref{thm: exponential convergence}, any triple $(\rho,P,\lambda) \in [0,1] \times \mathbb{S}_{+}^n \times \mathbb{R}_{+}$ that satisfies the matrix inequality in \eqref{eq: LMI} certifies an $\mathcal{O}(\rho^{2k})$ convergence rate for the algorithm. In particular, the fastest convergence rate can be found by solving the following optimization problem:
\begin{alignat}{2} \label{eq: algorithm analysis problem}
&\mathrm{minimize} \quad && \rho^2 \\ \nonumber 
&\text{subject to}  \quad && M(\theta,\rho,\lambda,P) \preceq 0 \nonumber \\ \nonumber
& && P \succeq 0, \ 0 \leq \rho \leq 1, \ \lambda \geq 0,
\end{alignat}
where $\theta$ is given (the algorithm parameters), and the decision variables are $\lambda, \rho, P$. Note that \eqref{eq: algorithm analysis problem} is a quasiconvex program since the constraint is affine in $\lambda$ and $P$ for a fixed $\rho$. We can therefore use a bisectioning search to find the smallest possible value of the decay rate $\rho$. This approach has been pursued in \cite{lessard2016analysis} using a different Lyapunov function. Furthermore, the corresponding matrix inequality for proximal variants of \eqref{eq: algorithm canonical form} has been developed in \cite{fazlyab2017dynamical}.

Note that for finding an $\epsilon$-accurate ($0<\epsilon\leq 1$) solution for the decay rate $\rho$, the computational complexity is at most $\mathcal{O}(q^3 \log_2(\epsilon^{-1}))$ where $q$, the total number of unknowns, is independent of the dimension $d$ of the optimization problem \eqref{eq: unconstrained problem} as we remark below. 
\begin{remark}\label{remark: structure}
	\normalfont We can often exploit some special structure in the matrices $A,B,C$, and $E$ {to reduce the dimension of \eqref{eq: LMI}. For many algorithms, these matrices are in the form
		\begin{align*}
		A=\bar{A} \otimes I_d, B=\bar{B} \otimes I_d, C = \bar{C} \otimes I_d, E=\bar{E} \otimes I_d,
		\end{align*}
		where now $\bar{A},\bar{B},\bar{C},$ and $\bar{E}$ are lower dimensional matrices independent of $d$ \cite[$\S$4.2]{lessard2016analysis}}. By selecting $P=\bar{P} \otimes I_d$, where $\bar{P}$ is a lower dimensional matrix, we can factor out all the Kronecker products $\otimes I_d$ from the matrices $M^0,M^1,M^2,M^3$ and make the dimension of the corresponding matrix inequality in \eqref{eq: LMI} independent of $d$. In particular, a multi-step method with $r \geq 1$ steps yields an $(r+1) \times (r+1)$ matrix inequality. For instance, the gradient method ($r=1$) and Nesterov's accelerated method ($r=2$) yield $2 \times 2$ and $3 \times 3$ matrix inequalities, respectively. 
\end{remark}

\begin{remark}[Non-monotone Algorithms] \normalfont
	We emphasize that in the development of Theorem \ref{thm: exponential convergence} although we require the Lyapunov function to decrease geometrically at each iteration (see \eqref{eq: Lyapunov exponential decay}), the resulting bound in \eqref{eq: convergence result} does not imply that the sequence of objective values is monotone, i.e., \eqref{eq: convergence rate} does not imply $f(x_{k+1}) \leq f(x_{k})$. This allows us to analyze the convergence properties of nonmonotone algorithms such as Nesterov's accelerated method. 
\end{remark}

\begin{remark}[Feasibility of \eqref{eq: algorithm analysis problem}] \normalfont
	The matrix inequality \eqref{eq: LMI} provides an upper bound on the true decay rate and, therefore, \eqref{eq: LMI} is sufficient for the exponential convergence result in \eqref{eq: convergence rate}. In other words, there might be an exponentially convergence algorithm for which \eqref{eq: LMI} is not feasible. Nevertheless, it has been shown in \cite{fazlyab2017analysis} that the bounds are  not conservative. For instance, for Nesterov's accelerated method, the rate bound obtained by solving \eqref{eq: algorithm analysis problem} turns out to be even better than the theoretical rate bound proved by Nesterov \cite{fazlyab2017analysis}.
\end{remark}


%

\section{Algorithm Synthesis} \label{section: Algorithm Synthesis}

We saw in the previous section that the exponential stability of a given first-order algorithm can be certified by solving an SDP feasibility problem. More precisely, given an algorithm in \eqref{eq: algorithm canonical form} with a prespecified value of $\theta$ (the tuning parameters of the algorithm), we can search for a suitable Lyapunov function and establish a rate bound for the algorithm by solving a quasiconvex program. A natural question to ask is whether we can leverage the same framework to do algorithm design. We formalize this problem as follows.
\begin{problem} \normalfont \label{design problem statement}
	Let $x_{\star} = \mathrm{argmin}_{x} \ f(x)$, where $f \in \mathcal{F}(m_f,L_f)$. Given a parameterized family of first-order methods given by \eqref{eq: algorithm canonical form}, tune the parameters $\theta$ of the algorithm, within a compact set $\Theta \subset \mathbb{R}^p$, such that the resulting algorithm converges at an $\mathcal{O}(\rho^{2k})$ rate to $x_{\star}$ with a minimal $\rho$.
\end{problem}
Using the result of Theorem \ref{thm: exponential convergence}, Problem \ref{design problem statement} can be formally written as the following \emph{nonconvex} optimization problem:
\begin{alignat}{2} \label{eq: algorithm design problem}
&\mathrm{minimize} \quad &&\rho^2 \\ \nonumber 
&\text{subject to} \quad && M(\theta,\rho,\lambda,P) \!\preceq \!0 \nonumber \\ \nonumber
& && P \succeq 0, \ 0 \leq \rho \leq 1, \ \lambda \geq 0,  \ \theta \in \Theta,
\end{alignat}
where the decision variables are now $\theta, P, \lambda$, and $\rho$. We recall from \eqref{eq: LMI matrices} that the parameter $\theta$ appears in \eqref{eq: algorithm design problem} through the matrices $A,B,C,E,M^1,M^2,M^3$. Intuitively, \eqref{eq: algorithm design problem} searches for the parameters $\theta$ of the algorithm for optimal performance (minimal $\rho$) while respecting the stability condition \eqref{eq: Lyapunov exponential decay}, which is imposed by the matrix inequality constraint in \eqref{eq: algorithm design problem}. 

Note that if we fix $\theta=\theta_0$, which means the algorithm parameters are given, then \eqref{eq: algorithm design problem} reduces to the quasiconvex program in \eqref{eq: algorithm analysis problem}. This suggests a natural but inefficient way to solve \eqref{eq: algorithm design problem}: We could do an exhaustive search over the parameter space $\Theta$, and solve \eqref{eq: algorithm analysis problem} to find the optimal $\rho^2$ for each value of $\theta$. We, therefore, need to solve a sequence of quasiconvex programs in order to find the optimal tuning. As an illustration, we implement this approach to tune the parameters of the Nesterov's accelerated method (the algorithm in \eqref{eq: universal example} with $\gamma=\beta$), where the tuning parameters are the stepsize $h$ and the momentum coefficient $\beta$. In Figure \ref{fig: robust_nesterov_contour}, we plot the level curves of the convergence factor $\rho$ as a function of $\theta = [h \ \beta]^\top$ in the region $\Theta = [0 \ 2/L_f] \times [0 \ 1]$. 

\begin{figure}
	\centering
	\includegraphics[width=\linewidth]{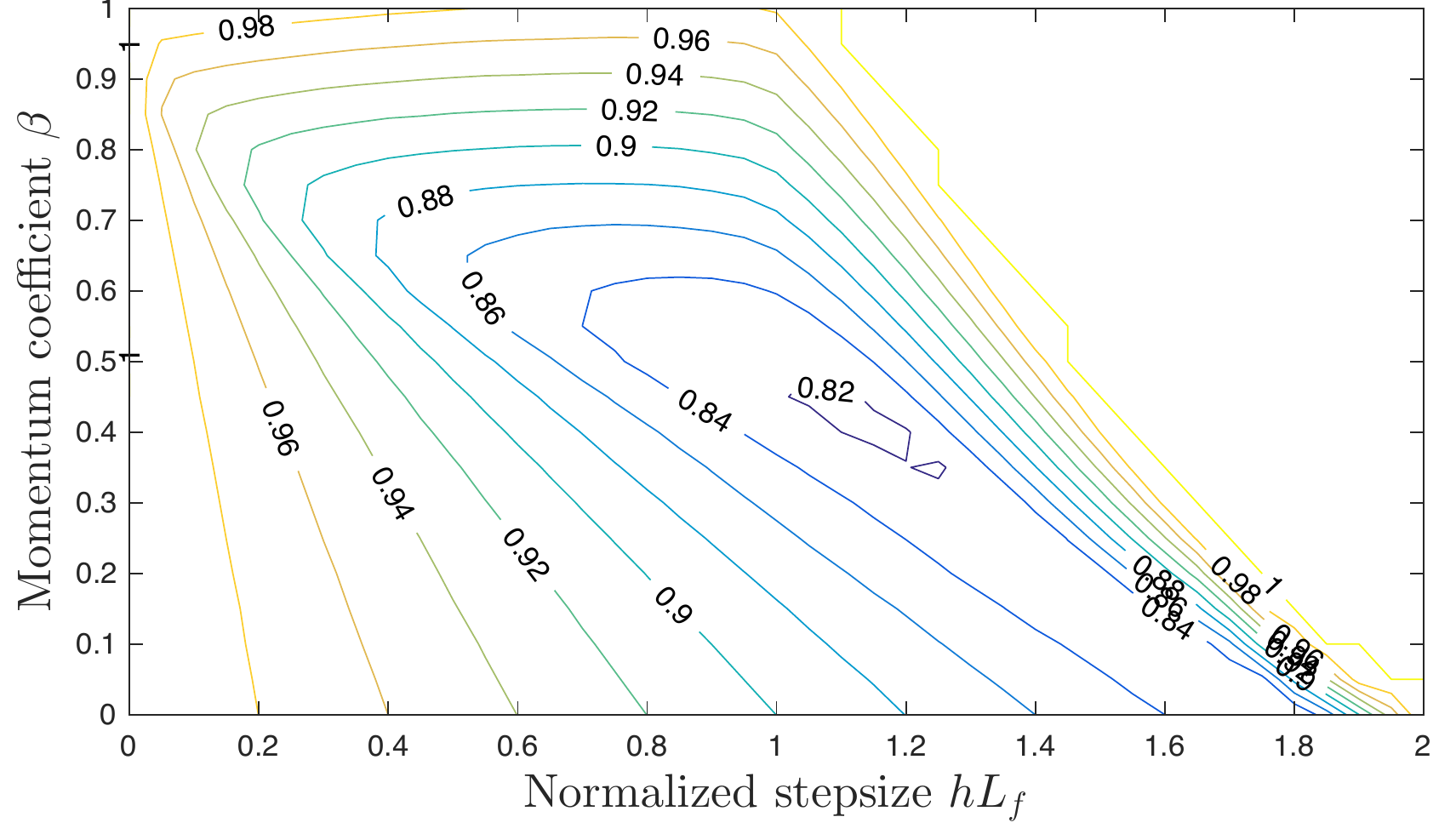}
	\caption{\small Plot of convergence rate $\rho$ of Nesterov's accelerated method as a function of stepsize $h$ and momentum parameter $\beta$ for $\kappa_f=10$.}
	\label{fig: robust_nesterov_contour}
\end{figure}

Note that, in general, the exhaustive search approach described above becomes prohibitively costly as the dimension and/or the granularity of the search space $\Theta$ increase. We therefore need an efficient way to solve \eqref{eq: algorithm design problem}. Although this problem is nonconvex, the special structure of the constraint set makes the problem tractable. To see this, we note that all the entries of the constraint matrix in \eqref{eq: algorithm design problem} are \emph{polynomial} functions of the decision variables. This matrix inequality constraint can be alternatively ``scalarized" and rewritten in terms of scalar polynomial inequalities (e.g., by considering minors, or coefficients of the
characteristic polynomial). The resulting optimization problem is of the form
\begin{alignat*}{2}
&\text{minimize} \quad  &&p\left(\mathbf{x}\right)\\
&\text{subject to } \quad && g_{i}\left(\mathbf{x}\right)\geq0, \ \text{for all }i=1,\ldots,m,
\end{alignat*}
where $\bx=[\rho^2 \ \theta^\top \ \text{vec}(P)^\top]^\top$ is the vector of decision variables, and $p(\bx)$ and $g_i(\bx), \ i=1,\ldots,m$ are all polynomials. It is here that we can draw a direct connection from the algorithm design problem in \eqref{eq: algorithm design problem} to polynomial optimization, which are tractable problems in many cases \cite{blekherman2012semidefinite}. We briefly introduce polynomial optimization next.



\subsection{Sum-of-Squares Programs} \label{subsection: SOS}

The main difficulty in solving problems involving polynomial constraints,
such as the one in \eqref{eq: algorithm design problem}, is the lack of efficient numerical methods
able to handle multivariate nonnegativity conditions. A computationally
efficient approach is to use sum-of-squares (SOS) relaxations \cite{Parrilo:Phd,Lasserre}.
In what follows, we introduce the basic results used in our derivations.
\begin{definition}
	A multivariate polynomial of degree $2d$ in $n$ variables with real
	coefficients, $p(x_{1},x_{2},...,x_{n})=p(\mathbf{x})\in\mathbb{R}[\mathbf{x}]{}_{n,2d}$,
	is a \emph{sum-of-squares} (SOS) if there exist polynomials $q_{1}(\mathbf{x}),...,q_{m}(\mathbf{x})\in\mathbb{R}[\mathbf{x}]_{n,d}$
	such that 
	\begin{equation}
	p(\mathbf{x})=\sum_{k=1}^{m}q_{k}^{2}(\mathbf{x}).\label{SOSeqn}
	\end{equation}
\end{definition}
We will denote the set of SOS polynomials in $n$ variables of degree at most $2d$ by $\Sigma_{n,2d}\left[\mathbf{x}\right]$. A polynomial being an SOS is sufficient to certify its global nonnegativity, since any $p\in\Sigma_{n,2d}\left[\mathbf{x}\right]$ satisfies $p(\mathbf{x})\geq0$ for all $\mathbf{x}\in\mathbb{R}^{n}$. Hence, $\Sigma_{n,2d}\subseteq P_{n,2d}$, where $P_{n,2d}\left[\mathbf{x}\right]$ is the set of nonnegative
polynomials in $\mathbb{R\left[\mathbf{x}\right]}_{n,2d}$. Given a polynomial $p\left(x\right)\in\mathbb{R}\left[\mathbf{x}\right]{}_{n,2d}$,
the existence of an SOS decomposition of the form (\ref{SOSeqn}) is
equivalent to the existence of a positive semidefinite matrix $Q\in\mathbb{S}_{+}^{{n+d \choose d}}$,
called the \emph{Gram matrix}, such that
\begin{equation}
p(\mathbf{x})=\left[\mathbf{x}\right]_{d}^{T}Q\left[\mathbf{x}\right]_{d},\label{Q}
\end{equation}
where $\left[\mathbf{x}\right]_{d}=\left[1,x_{1},\ldots,x_{n},x_{1}^{2},x_{1}x_{2},\ldots,x_{n}^{d}\right]$
is the vector of all ${n+d \choose d}$ monomials in $x_{1},\ldots,x_{n}$
of degree at most $d$. Notice that the equality constraint in (\ref{Q})
is affine in the matrix $Q$, since the expansion of the right-hand
side results in a polynomial whose coefficients depend affinely on
the entries of $Q$ and must be equal to the corresponding coefficients
of the given polynomial $p\left(\mathbf{x}\right)$. Hence, finding
an SOS decomposition is computationally equivalent to finding a positive
semidefinite matrix $Q$ subject to the affine constraint in (\ref{Q}),
which is a semidefinite program \cite{Parrilo:Phd,SOSTOOLS}.

Using the notion of sos polynomials, we can now define the class of
\emph{sum-of-squares programs} (SOSP)\emph{.} An SOSP is an optimization
program in which we maximize a linear function over a feasible set
given by the intersection of an affine family of polynomials and the
set $\Sigma\left[\mathbf{x}\right]$ of SOS polynomials in $\mathbb{R}\left[\mathbf{x}\right]$,
as described below \cite{Parrilo:book}:
\begin{align*}
&\text{maximize }  \ \ b_{1}y_{1}+\cdots+b_{m}y_{m}\\
&\text{subject to } \ \ p_{i}\left(\mathbf{x};\mathbf{y}\right)\in\Sigma\left[\mathbf{x}\right], \quad \text{for all }i=1,\ldots,k,
\end{align*}
where $\by = [y_1 \cdots y_m]^\top$ is the vector of decision variables, $p_{i}\left(\mathbf{x};\mathbf{y}\right)=c_{i}\left(\mathbf{x}\right)+a_{i1}\left(\mathbf{x}\right)y_{1}+\cdots+a_{im}\left(\mathbf{x}\right)y_{m}$,
and $c_{i},\,a_{ij}$ are given multivariate polynomials in $\mathbb{R}\left[\mathbf{x}\right]$. Note that in the above optimization problem, $\bx$ is the vector of indeterminates and not the decision variables. Despite their generality, it can be proved that SOSPs are equivalent
to SDPs; hence, they are convex programs and can be solved in polynomial
time \cite{SOSTOOLS}. In recent years, SOSPs have been used as convex
relaxations for various computationally hard optimization and control
problems (see, for example, \cite{Parrilo:Phd,sdprelax,Lasserre,CAND,majumdar2012control}
and the volume \cite{GarulliHenrion}).

%
The notion of positive definiteness and sum-of-squares of scalar-valued polynomials can be extended to polynomial matrices, i.e., matrices with entries in $\mathbb{R}[\bx]$. The definition of an sos matrix is as follows \cite{ParriloGatermann}.
\begin{definition}
	A symmetric polynomial matrix $P(\bx) \in \mathbb{R}[x]^{m \times m}, \ \bx \in \mathbb{R}^n$, is an sos matrix if there exists a polynomial matrix $M(\bx) \in \mathbb{R}^{s \times m}$ for some positive integer $s$, such that $P(\bx)=M^\top(\bx)M(\bx)$.
\end{definition}
Since an $m \times m$ matrix is simply a representation of an $m$-variate quadratic form, we can always interpret an sos matrix in terms of a polynomial with $m$ additional variables. The following lemma makes this precise.

\begin{lemma} \label{lemma: SOSM}
	Consider a symmetric matrix with polynomial entries $\mathbf{P}(\mathbf{x})\in\mathbb{R}[\mathbf{x}]^{m\times m}$,
	and let $\mathbf{z}=[z_{1},\ldots,z_{m}]^{T}$ be a vector of indeterminates.
	Then $\mathbf{P}(\mathbf{x})$ is a \emph{sum-of-squares matrix} (SOSM)
	if $\mathbf{z}^{T}\mathbf{P}(\mathbf{x})\mathbf{z}$ is an SOS polynomial
	in $\mathbb{R}[\mathbf{x},\mathbf{z}]$.
\end{lemma}
\noindent Obviously, a polynomial matrix $\mathbf{P}\left(\mathbf{x}\right)$
being SOSM provides an explicit certificate for $\mathbf{P}\left(\mathbf{x}\right)$
being positive semidefinite for all $\mathbf{x}\in\mathbb{R}^{n}$.

\subsection{Polynomial Optimization Problems}
One application of SOSP is the global optimization of a polynomial $p(\bx)$. To this end, rather than directly computing a minimizer $\bx_{\star}$ of $p(\bx)$, we instead focus on obtaining the best possible lower bound on its \emph{optimal value} $p(\bx_{\star})$. This viewpoint is based on the observation that a real-valued number $\gamma$ is a global lower bound of $p(\bx)$ if and only if the polynomial $p(\bx)-\gamma$ is nonnegative for all $\bx$. The best lower bound on $p(\bx)$ is thus obtained by solving the optimization problem
\begin{align} \label{eq: best lower bound}
p_{\star} = \mathrm{sup}_{\gamma} \ \ \gamma \quad \text{subject to} \quad p(\bx)-\gamma \geq 0, 
\end{align}
where the decision variable is now $\gamma$ and the original decision variable $\bx$ acts as an indeterminate. Observe that \eqref{eq: best lower bound} is a convex optimization problem with infinitely many constraints. By replacing the nonnegativity condition with an sos constraint, we obtain the following optimization problem
\begin{align*}
p_{sos} = \mathrm{sup}_{\gamma} \ \gamma \quad \text{subject to} \quad p(\bx)-\gamma \in \Sigma[\bx].
\end{align*}
Note that for multivariate polynomials, nonnegativity and sos are not equivalent. More precisely, we have the inclusion $\Sigma_{n,2d}\subseteq P_{n,2d}$. Therefore, since the feasible set of the second problem is a subset of the feasible set of the first problem, we have the inequality $p_{sos} \leq p_{\star}$. However, for relatively small problems, we often have $p_{sos}=p_{\star}$, i.e., there is no loss of optimality by using sos relaxations. Nevertheless, even in those situations where $p_{sos} < p_{\star}$, we can improve the lower bound $p_{sos}$ by producing stronger sos conditions \cite{blekherman2012semidefinite}.

In our particular application, we need to optimize a multivariate polynomial over a set described by polynomial inequalities:
%
%
\begin{align} \label{eq: constrained polynomial opt}
&\text{minimize}_{\bx \in \mathbb{R}^d} && p\left(\mathbf{x}\right) \nonumber \\
&\text{subject to } && g_{i}\left(\mathbf{x}\right)\geq0, \ \text{for all }i=1,\ldots,m,
\end{align}
where $p\left(\mathbf{x}\right)$, $g_{i}\left(\mathbf{x}\right)\in\mathbb{R}\left[\mathbf{x}\right]$
for all $i=1,\ldots,m$. Similar to the unconstrained case, we can instead find the best lower bound on $p(\bx)$ on the constraint set, as follows:
\begin{align*}
&\text{maximize}_{\gamma} && \gamma \\
&\text{subject to } && p(x)-\gamma \geq 0 \quad \text{for all } x \in S,
\end{align*}
where $S = \{\bx \in \mathbb{R}^n \colon  \ g_{i}\left(\mathbf{x}\right)\geq0, \ \text{for all }i=1,\ldots,m\}$. The corresponding sos relaxation requires $p(\bx)$ to be sos on $S$. Recalling the formal similarity with weak duality and Lagrange multipliers, it is natural to consider the following decomposition for $p(\bx)$:
\begin{align*}
p(\bx) = s_{0}\left(\mathbf{x}\right)+\sum_{i=1}^{m}s_{i}\left(\mathbf{x}\right)g_{i}\left(\mathbf{x}\right) \quad \bx \in S,
\end{align*}
where $s_0(\bx)$ and $s_i(\bx)$ are sos polynomials that are determined by matching the coefficients of the left- and right-hand side. This particular decomposition implies that for any $\bx_0 \in S$, we have $g_i(\bx_0) \geq 0$ for all $i$ and therefore, the condition $p(\bx_0) \geq 0$ is automatically satisfied. Considering this decomposition, we obtain the SOCP problem
\begin{align*}
\text{maximize}_{\gamma} & \ \gamma  \ \text{subject to} \ s_{0}\left(\mathbf{x}\right)\!+\!\sum_{i=1}^{m}s_{i}\left(\mathbf{x}\right)g_{i}\left(\mathbf{x}\right)\!-\!\gamma \in \Sigma[\bx].
\end{align*}
Following the discussion in $\S$\ref{subsection: SOS}, the above problem is an SOCP and can be converted to an SDP.


%

\section{Numerical Simulations}

In this section, we validate the proposed approach on two examples: the gradient method and Nesterov's accelerated method. There are several software packages that convert polynomial optimization problems into a hierarchy of SDP relaxations and then call an external solver to solve them. In this paper, we use the software Gloptipoly3 \cite{henrion2009gloptipoly}, which is oriented toward global polynomial optimization problems of the form \eqref{eq: constrained polynomial opt}.

\subsection{The Gradient Method}
Consider the gradient method of Example \ref{example: gradient method}. By choosing $P=p I_d$ ($p \geq 0$), we can apply the dimensionality reduction outlined in Remark \ref{remark: structure} and reduce the dimension of the LMI. After dimensionality reduction, the matrices $M^0, M^1, M^2$, and $M^3$ in the LMI \eqref{eq: LMI} read as
\begin{subequations}
	\begin{align*} 
	M^0 &= \begin{bmatrix}
	1\!-\! \rho^2& -h\\  -h & h^2
	\end{bmatrix} \times p     \\ \nonumber
	M^{1}&\!=\begin{bmatrix}
	0 & 0 \\ 0 &  \frac{1}{2}(L_f h^2-2h)
	\end{bmatrix} \\ \nonumber 
	M^{2}&= \!\begin{bmatrix}
	-\frac{m_f}{2} \!&\! \frac{1}{2} \\  \frac{1}{2} \!&\!  \frac{1}{2}(L_f h^2-2h)
	\end{bmatrix} \\ \nonumber
	M^{3}&= \begin{bmatrix}
	-2m_fL_f & m_f+L_f \\ m_f+L_f  & -2
	\end{bmatrix}.
	\end{align*}
	By substituting these matrices back in  \eqref{eq: LMI}, we obtain the following matrix inequality constraint:
	\begin{align*}
	M = M^0+\!\rho^2 M^1 \!+\! (1\!-\!\rho^2) M^2 \!+\! \lambda M^3 \preceq 0,
	\end{align*}
	where the entries $M_{ij},\ 1 \leq i,j \leq 2$ of $M$ are all polynomials functions of $\rho^2,h,\lambda, p$ and are given by
	\begin{align*}
	M_{11} &= (p-\frac{m_f}{2})(1-\rho^2)-2m_fL_f\lambda \\
	M_{12} &= M_{21} = -hp-\frac{\rho^2}{2}+\frac{1}{2}+\lambda (m_f+L_f) \\
	M_{22} &= h^2+\dfrac{1}{2}(L_fh^2-2h)-2\lambda.
	\end{align*}
\end{subequations}
Recalling that the condition $M \preceq 0$ is equivalent to $\mathrm{Tr}(M) \leq 0$ and $\mathrm{det}(M) \geq 0$, the design problem in \eqref{eq: algorithm design problem} for the gradient method is equivalent to the following polynomial optimization problem:
\begin{alignat}{2} \label{eq: gradient method poly}
&\mathrm{minimize} \ &&\rho^2 \\ \nonumber
&\mbox{subject to} \ &&-M_{11}-M_{22} \geq 0 \\ \nonumber
& &&M_{11}M_{22}-(M_{12})^2 \geq 0 \\ \nonumber
& && p\geq 0, \ h \geq 0, \ \lambda \geq 0, \nonumber
\end{alignat}
with decision variables $\rho^2,h,\lambda$, and $p$. In Figure \ref{fig: rho_gradient}, we plot the optimal rate bound $\rho_{sos}$, obtained by solving an sos relaxation of \eqref{eq: gradient method poly}, for various values of the condition number $\kappa_f$. We also plot the analytical rate outlined in Table \ref{table}. We observe that the rate obtained from the sos formulation coincides with the analytical rate, which is known to be tight.

\begin{figure}[h]
	\centering
	\includegraphics[width=\linewidth]{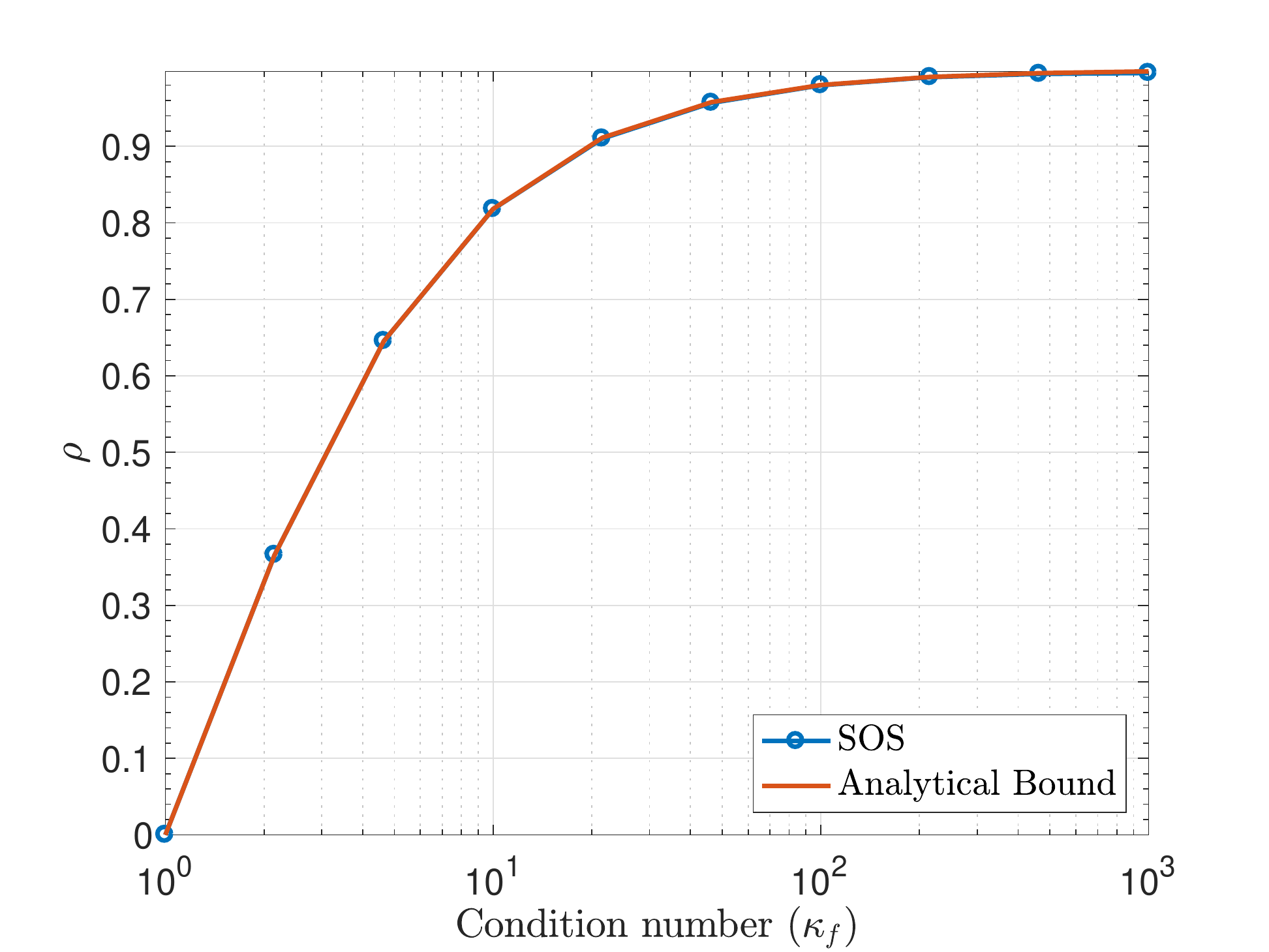}
	\caption{\small Convergence rate of the gradient method obtained by the sos technique and the analytical rate $\rho=\frac{\kappa_f-1}{\kappa_f+1}$.}
	\label{fig: rho_gradient}
\end{figure}

\subsection{Nesterov's Accelerated Method}

Consider the algorithm of Example \ref{example: nesterov heavy-ball method} with $\gamma=\beta$, which corresponds to Nesterov's accelerated method. The state-space matrices of this algorithm are given in \eqref{eq: nesterov heavy-ball matrices}. By applying the dimensionality reduction of Remark \ref{remark: structure}, we arrive at the following matrices that appear in the LMI \eqref{eq: LMI matrices}:
\begin{align} \label{eq: Nesterov method LMI matrices}
M^{0}&\!=\!  \begin{bmatrix}
A^\top P A \!-\! \rho^2 P& A^\top P B\\  B^\top P A& B^\top P B
\end{bmatrix} \\ \nonumber
M^{1}&\!=\! \begin{bmatrix}
-\frac{1}{2}m_f \beta^2 & \frac{1}{2}m_f\beta^2 & -\frac{1}{2}\beta \\  \frac{1}{2}m_f\beta^2 &  -\frac{1}{2}m_f \beta^2 & \frac{1}{2}\beta\\ -\frac{1}{2}\beta & \frac{1}{2}\beta & \frac{1}{2}L_f h^2-h
\end{bmatrix} \\ \nonumber
M^{2}&\!=\! \begin{bmatrix}
-\frac{1}{2}m_f\beta^2 & \frac{1}{2}m_f\beta(\beta+1) & -\frac{1}{2}\beta \\  \frac{1}{2}m_f\beta(\beta+1) &  -\frac{1}{2}m_f(\beta+1)^2 & \frac{1}{2}(\beta+1)  \\ -\frac{1}{2}\beta & \frac{1}{2}(\beta+1) & \frac{1}{2}L_f h^2-h
\end{bmatrix} \\ \nonumber
M^{3} &\!=\! \begin{bmatrix}
-\beta & 0 \\ (1+\beta) & 0 \\ 0 & 1
\end{bmatrix} Q_f \begin{bmatrix}
-\beta   & (1+\beta) & 0 \\ 0 & 0 & 1
\end{bmatrix},
\end{align}
where $Q_f$ is given in \eqref{eq: quadratic matrix} and $P$ is now a $2\times 2$ positive semidefinite matrix. By defining $\bx=[\rho^2 \ h \ \beta \ \lambda \  \mbox{vec}(P)^\top \  ]^\top \in \mathbb{R}^{7}$, $p(\bx)=\rho^2$, and $M(\bx)=M^0+\rho^2 M^1+(1-\rho^2)M^2+\lambda M^3 \in \mathbb{S}^{3}$, the corresponding optimization problem can be written as
\begin{align} \label{eq: Nesterov SOS}
\mbox{minimize} \ p(\bx) \quad \mbox{subject to } -M(\bx) \succeq 0.
\end{align}
Using Lemma \ref{lemma: SOSM} we can convert the polynomial matrix inequality $-M(\bx) \succeq 0$ to a single polynomial inequality in a higher-dimensional space (see Lemma \ref{lemma: SOSM}). Here we scalarize the positive semidefiniteness constraint by considering the principal minors (Sylvester's criterion for positive semidefinite matrices). Furthermore, we fix the stepsize $h$ to the value $1/L_f$ for solving the sos relaxation. In other words, we optimize $\rho^2$ over $(\beta,\lambda,P)$.

In Figure \ref{fig: rho_Nesterov}, we plot the optimal decay rate $\rho_{sos}$, obtained by solving an sos relaxation of \eqref{eq: Nesterov SOS}, for various values of $\kappa_f$. We observe that the obtained rate is slightly better than Nesterov's rate and worse than the theoretical lower bound, which is achieved by the Heavy-ball method on quadratic objective functions--see Table \ref{table}.


\begin{figure}
	\centering
	\includegraphics[width=\linewidth]{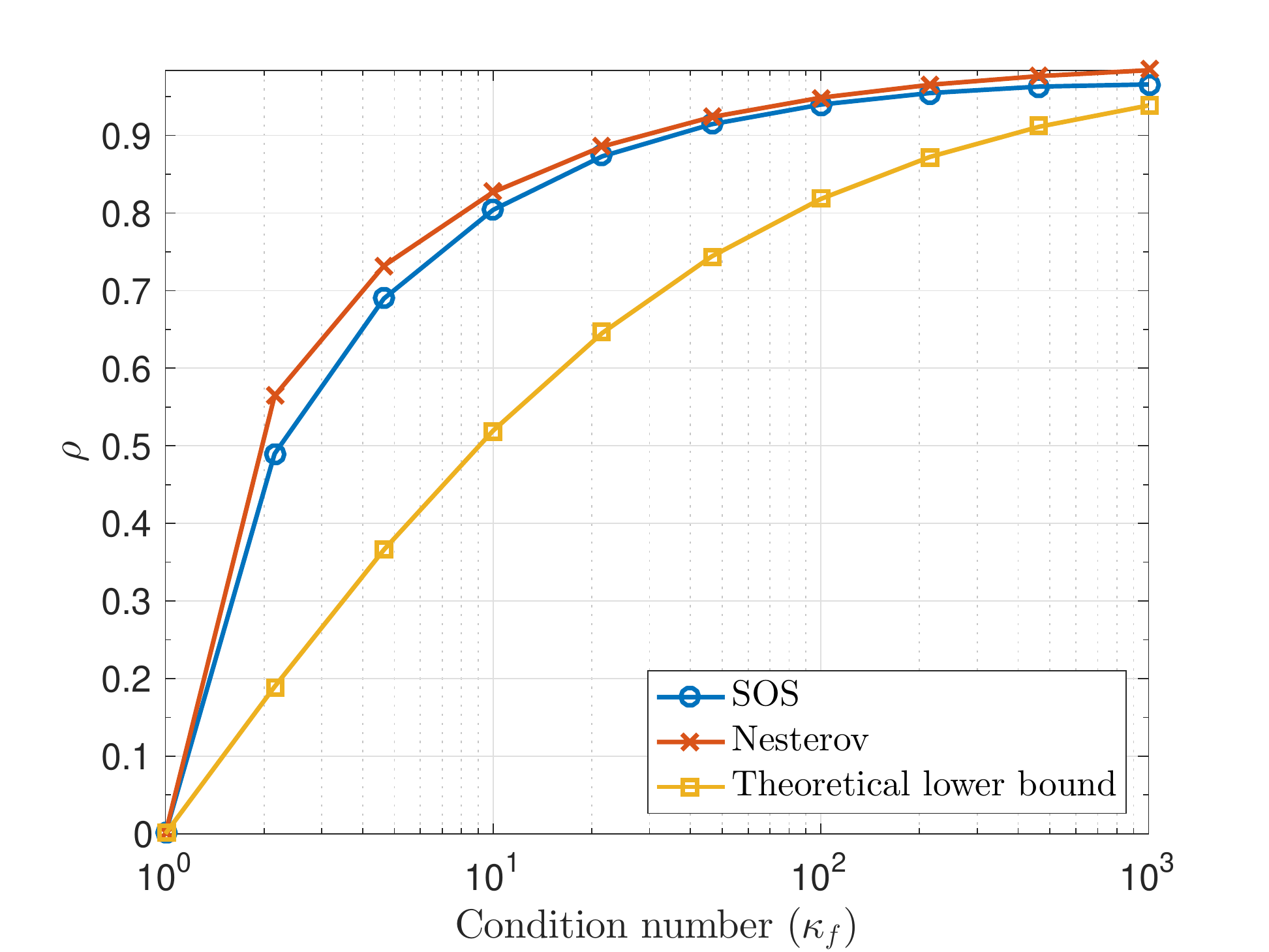}
	\caption{\small Comparison of rate bounds of Nesterov's accelerated method for various values of the condition number and different parameter tunings. The blue curve is obtained by tuning the parameters using the developed sos framework. The red curve is for the parameter tuning proposed by Nesterov \cite{nesterov1983method} (see also Table \ref{table}). Finally, the yellow curve is the theoretical lower bound for all first-order methods, which is achieved by the Heavy-ball method on quadratic objective functions--see Table \ref{table}.}
	\label{fig: rho_Nesterov}
\end{figure}

\section{Conclusions}
We have considered the problem of designing first-order iterative optimization algorithms for solving smooth and strongly convex optimization problems. By using a family of parameterized nonquadratic Lyapunov functions, we presented a polynomial matrix inequality as a sufficient condition for exponential stability of the algorithm. All the entries of this matrix have a polynomial dependence on the unknown parameters, which are the parameters of the algorithm, the parameters of the Lyapunov function, and the exponential decay rate. We then formulated a polynomial optimization problem to search for the optimal convergence rate subject to the polynomial matrix inequality. Finally, we proposed a sum-of-squares relaxation to solve the resulting design problem. We illustrated the proposed approach via numerical simulations. 

\section*{Acknowledgments}

We would like to thank the anonymous reviewers for their useful comments on the earlier draft of this paper. We would also like to thank Amir Ali Ahmadi from Princeton University for his helpful discussion and suggestions.

\bibliographystyle{ieeetr}
\bibliography{Refs}

\end{document}